\documentclass[12pt]{amsart}

\usepackage{amssymb,latexsym}
\usepackage{enumerate}
\usepackage{hyperref}

\makeatletter
\@namedef{subjclassname@2010}{%
  \textup{2010} Mathematics Subject Classification}
\makeatother

\newtheorem {theorem}{Theorem}[section]
\newtheorem {proposition}{Proposition}[section]

\newtheorem {lemma}{Lemma}[section]
\newtheorem {example}{Example}[section]
\newtheorem {definition}{Definition}[section]
\newtheorem {remark}{Remark}[section]

\DeclareMathOperator\KKT{KKT}
\DeclareMathOperator\grad{grad}

\begin{document}

\baselineskip=17pt

\title[Representations of non-negative polynomials]{Representations of non-negative polynomials via critical ideals}

\author[D. T. Hiep]{Dang Tuan Hiep}

\address{Department of Mathematics, University of Dalat, 01 Phu Dong Thien Vuong, Da Lat, Vietnam}

\email{dangtuanhiep210@yahoo.com}

\address{{\it Current address:} Dipartimento di Matematica, Universit\`{a} degli Studi di Bari Aldo Moro, Via E. Orabona, 4 - 70125, Bari, ITALY}

\email{dang@dm.uniba.it}



\begin{abstract}
This paper studies the representations of a non-negative polynomial $f$ on a non-compact semi-algebraic set $K$ modulo its critical ideal. Under the assumptions that the semi-algebraic set $K$ is regular and $f$ satisfies the boundary Hessian conditions (BHC) at each zero of $f$ in $K$, we show that $f$ can be represented as a sum of squares (SOS) of real polynomials modulo its critical ideal if $f\ge 0$ on $K$. In particular, we focus on the polynomial ring $\mathbb R[x]$.
\end{abstract}

\maketitle

\section{introduction}

We know that a polynomial in one variable $f(x)\in\mathbb R[x]$ satisfies $f(x)\ge 0$, for all $x\in\mathbb R$, then $f(x)=\sum_{i=1}^mg^2_i(x)$, where $g_i(x)\in\mathbb R[x]$, i.e., $f$ is a sum of squares in $\mathbb R[x]$ (SOS for short). However, in multi-variate cases, this is not true. A counterexample was given by Motzkin in 1967. If $f(x,y)=1+x^4y^2+x^2y^4-3x^2y^2$, then $f(x,y)\ge 0$, for all $x,y\in \mathbb R$. But $f$ is not a SOS in $\mathbb R[x,y]$. To remedy that, we will consider the polynomials that are positive on $K$, where $K$ is a semi-algebraic set in $\mathbb R^n$. For example, Schm\"{u}dgen's theorem \cite{Schmudgen1991} states that for a compact semi-algebraic set, every strictly positive polynomial belongs to the corresponding finitely generated preordering. Afterward, Putinar \cite{Putinar1993} simplified this representation under an additional assumption by using the quadratic module instead of the preordering. However, these results of Schm\"{u}dgen and Putinar have two restrictions. Firstly, the polynomials are positive, not merely non-negative. Secondly, $K$ must be a compact semi-algebraic set. Hence we seek to identify the representations of the non-negative polynomials on the non-compact semi-algebraic sets.

In \cite{NDS2006}, the authors presented a representation of the non-negative polynomials on the whole space modulo their gradient ideals. Afterward, in \cite{NDP2007}, the authors proved a similar representation on the arbitrary semi-algebraic sets. These results were achieved under the condition of the corresponding ideals must be radical. However, it is not simple to check this condition. In order to overcome such limitation, in \cite{Marshall2009}, Marshall considered another condition - the boundary Hessian condition (BHC). He proved that the result in \cite{NDS2006} still held true if the radical condition  is replaced by the BHC condition. In \cite{Dang2009}, the author presented an extension of theorem 2.1 in \cite{Marshall2009} in the same way that the result in \cite{NDP2007} was the extension of the corresponding result in \cite{NDS2006}.

However, in \cite{Dang2009} and \cite{NDP2007} the authors considered a larger polynomial ring $\mathbb R[x,\lambda]$, i.e., they added Lagrange multipliers to the representations. This paper will help us overcome this. We will present the representations of the non-negative polynomials via their critical ideals. In particular, we focus on the polynomial ring $\mathbb R[x]$.

\section{Preliminaries}
In this section, we present some notions and results from algebraic geometry and real algebra needed for our discussions. The readers may consult \cite{BCR}, \cite{Cox1997}, and \cite{PD2004} for more details.

Throughout this paper, denote by $\mathbb R[x]$ the ring of polynomials in $x=(x_1,\ldots, x_m)$ with real coefficients. Given an ideal $I\subseteq\mathbb R[x]$, define its {\it complex variety} to be the set
$$V(I)=\{x\in\mathbb C^m\mid p(x)=0,\ \forall p\in I\},$$
and its {\it real variety} to be
$$V^{\mathbb R}(I)=V(I)\cap \mathbb R^m.$$

A nonempty variety $V=V(I)\subseteq\mathbb C^m$ is {\it irreducible} if there do not exist two proper subvarieties $V_1,V_2\subset V$ such that $V=V_1\cup V_2$. The readers should note that in this paper, ``irreducible'' means that the set of complex zeros cannot be written as a proper union of subvarieties defined by real polynomials.

Given any ideal $I$ of $\mathbb R[x]$, its radical ideal $\sqrt{I}$ is defined to be the following ideal:
$$\sqrt{I}=\{q\in\mathbb R[x]\mid q^l\in I\ \text{for some}\ l\in\mathbb N\}.$$
Clearly, $I\subseteq\sqrt{I}$; $I$ is a {\it radical ideal} if $\sqrt{I}=I$. As usual, for a variety $V\subseteq\mathbb C^m$, $I(V)$ denotes the ideal in $\mathbb C[x]$ of polynomials vanishing on $V$. We will write $I^{\mathbb R}(V)$ for the ideal $I(V)\cap\mathbb R[x]$.

We need versions of the Nullstellens\"{a}tz for varieties defined by polynomials in $\mathbb R[x]$. The following two theorems are normally stated for ideals in $\mathbb C[x]$; however, keeping in mind that $V(I)$ lies in $\mathbb C^m$, they hold as stated for ideals in $\mathbb R[x]$.
\begin{theorem}[\cite{Cox1997}]\label{weaknullstellensatz}
If $I$ is an ideal in $\mathbb R[x]$ such that $V(I)=\emptyset$, then $1\in I$.
\end{theorem}
\begin{theorem}[\cite{Cox1997}]\label{strongnullstellensatz}
If $I$ is an ideal in $\mathbb R[x]$, then $I^{\mathbb R}(V(I))=\sqrt{I}$.
\end{theorem}

Let $g_1,\ldots,g_s\in\mathbb R[x]$. We define the {\it preordering} generated by $g_1,\ldots,g_s$ as follows:
$$P=\biggl\{\sum_{e\in\{0,1\}^s}\sigma_eg_1^{e_1}\ldots g_s^{e_s}\biggr\},$$
where $e=(e_1,\ldots,e_s)\in\{0,1\}^s$ and $\sigma_e$ are sums of squares of polynomials in $\mathbb R[x]$.\\
We also define the {\it semi-algebraic set} generated by $g_1,\ldots,g_s$ as follows:
$$K=\{x\in\mathbb R^n\mid g_i(x)\ge 0, i=1,\ldots,s\}.$$

\begin{definition}[see \cite{NW1999}, Definition 12.1]
For each $x\in\mathbb R^n$, let $J_x$ be the set of indices $j$ for which $g_j$ vanishes at $x$. The semi-algebraic set $K$ is called {\it regular}, if for each $x\in K$, the vectors $\nabla g_j(x), j\in J_x$, are linearly independent.
\end{definition}

Throughout this paper, we always assume that the semi-algebraic set $K$ is regular.

\section{The critical variety}

\begin{definition}
The {\it critical variety of $f$ on $K$} is defined as follows:
\begin{eqnarray*}
C(f,K):=\{x\in\mathbb R^n &\mid& \ \textrm{there exist real numbers $\lambda_i$ such that}\\
&&\nabla f(x)-\sum_{i=1}^s\lambda_i\nabla g_i(x)=0,\\
&&\lambda_ig_i(x)=0, i=1,\ldots,s\}.
\end{eqnarray*}
\end{definition}
\begin{remark}\rm
\begin{itemize}
\item[]
\item[(i)] In the global case, i.e., when the semi-algebraic set $K$ is the whole space $\mathbb R^n$, we have
$$C(f,K)=\{x\in\mathbb R^n\mid\nabla f(x)=0\},$$
which is the {\it real gradient variety} of $f$ (see \cite{NDS2006}).
\item[(ii)] Consider the projection $\pi\colon\mathbb R^n\times\mathbb R^s\to\mathbb R^n, (x,\lambda)\mapsto x$, where variables $\lambda=(\lambda_1,\ldots,\lambda_s)$ are Lagrange multipliers. Then $C(f,K)=\pi(V_{\KKT})$, here
\begin{eqnarray*}
V_{\KKT}:=\{(x,\lambda)\in\mathbb R^n\times\mathbb R^s &\mid& \nabla f(x)-\sum_{i=1}^s\lambda_i\nabla g_i(x)=0,\\
&& \lambda_ig_i(x)=0, i=1,\ldots,s\},
\end{eqnarray*}
is the {\it real KKT variety of $f$ on $K$} (see \cite{NDP2007}).
\end{itemize}
\end{remark}
In this section, we will study the properties of the critical variety $C(f,K)$.
\begin{proposition}\label{var1}
The following statements hold true
\begin{itemize}
\item[(i)] $C(f,K)=C(f+a,K)$, for all $a\in\mathbb R$.
\item[(ii)] If $f$ attains its infimum at $x^*\in K$, then $x^*\in C(f,K)$.
\end{itemize}
\end{proposition}
\begin{proof}
\begin{itemize}
\item[]
\item[(i)] We see clearly that $\nabla f=\nabla (f+a)$, for all $a\in\mathbb R$. Then, by definition of the critical variety, we have $C(f,K)=C(f+a,K)$, for all $a\in\mathbb R$.
\item[(ii)] By Karush-Kuhn-Tucker theorem (see e.g. \cite{NW1999}), if $f$ attains its infimum at $x^*\in K$, then there exist $\lambda^*_0,\lambda^*_1,\ldots,\lambda^*_s$ at least one of which is different from zero, such that
    $$\lambda^*_0\nabla f(x^*)-\sum_{i=1}^s\lambda^*_i\nabla g_i(x^*)=0,$$
    $$\lambda^*_ig_i(x^*)=0, i=1,\ldots,s.$$
    Since $K$ is regular, then we can choose $\lambda^*_0=1$.\\
    Thus $x^*\in C(f,K)$.
\end{itemize}

\end{proof}
We will use the following notations in the remainder of the paper.
\begin{definition}
For each subset $J$ of $\{1,\ldots,s\}$, we consider the polynomial
$$g_J(x):=\left\{\begin{array}{lll}
\prod_{j\in J}g_j(x) & , & J\ne\emptyset,\\
1 & , & J=\emptyset.
\end{array}\right.$$
If $J=\{j_1\ldots,j_k\}$, we will denote by $h_J\in\mathbb R[x]$ the following polynomial
$$h_J(x):=\det(A_J(x)A_J^T(x)),$$
where
$$A_J(x):=\left(\begin{array}{cccc}
\frac{\partial f}{\partial x_1} & \frac{\partial f}{\partial x_2} & \cdots & \frac{\partial f}{\partial x_n}\\
\frac{\partial g_{j_1}}{\partial x_1} & \frac{\partial g_{j_1}}{\partial x_2} & \cdots & \frac{\partial g_{j_1}}{\partial x_n}\\
\vdots & \vdots & \cdots & \vdots \\
\frac{\partial g_{j_k}}{\partial x_1} & \frac{\partial g_{j_k}}{\partial x_2} & \cdots & \frac{\partial g_{j_k}}{\partial x_n}
\end{array}\right)$$
is a $(k+1)\times n$-matrix. Observe that $h_J(x)=0$ if and only if the vectors $\nabla f, \nabla g_j, j\in J$ are linearly dependent.
\end{definition}
\begin{proposition}\label{var2}
The critical variety $C(f,K)$ is an algebraic set. More precisely we have
$$C(f,K)=\{x\in\mathbb R^n\mid g_J(x)h_{J^c}(x)=0, \forall J\subseteq\{1,\ldots,s\}\},$$
where we use the notation $J^c:=\{1,\ldots,s\}\backslash J$.
\end{proposition}
\begin{proof}
The proof is similar as that of Proposition 3.1 in \cite{HP2009} and therefore is omitted here.
\end{proof}

\section{Boundary Hessian Conditions, gradient ideals and KKT ideals} \label{BHC}

We say $f$ satisfies the BHC (boundary Hessian conditions) at the point $x^*$ in $K$ if there are some $k \in \{1,\ldots,n\}$, and $v_1,...,v_k \in \mathbb N$ with $1 \le v_1 < ... < v_k \le s$ such that $g_{v_1},\ldots,g_{v_k}$ are parts of a system of local parameters at $x^*$, and the standard sufficient conditions for a local minimum of $f|_L$ at $x^*$ hold, where $L$ is the subset of $\mathbb R^n$ defined by $g_{v_1}(x)\ge 0,\ldots,g_{v_k}(x)\ge 0$. This means that if $t_1,\ldots,t_n$ are local parameters at $x^*$ chosen so that $t_i=g_{v_i}$ for $i\le k$, then in the completion $\mathbb R[[t_1,\ldots,t_n]]$ of $\mathbb R[x]$ at $x^*$, $f$ decomposes as $f=f_0+f_1+f_2+\cdots$ (where $f_i$ is homogeneous of degree $i$ in the variables $t_1,\ldots,t_n$ with coefficients in $\mathbb R$), $f_1=a_1t_1+\cdots+a_kt_k$ with $a_i>0, i=1,\ldots,k$, and the $(n-k)$-dimensional quadratic form $f_2(0,\ldots,0,t_{k+1},\ldots,t_n)$ is positive definite.

\begin{theorem}[Marshall \cite{Marshall2009}]\label{Marshall}
If $f$ satisfies the BHC at each zero of $f$ in $K$, then $f\in P+\langle f^2\rangle$.
\end{theorem}

\begin{example}\rm
Let $f,g_1\in\mathbb R[x,y,z]$ be given by $$f(x,y,z)=x; g_1(x,y,z)=x-y^2-z^2.$$ Then $$K=\{(x,y,z)\in\mathbb R^3\mid z-y^2-z^2\ge 0\}.$$
Clearly, $f\ge 0$ on $K$, and the unique zero of $f$ in $K$ occurs at $(0,0,0)$. Furthermore, $f$ satisfies the BHC at $(0,0,0)$. Indeed, let $t_1=g_1=x-y^2-z^2, t_2=y$ and $t_3=z$. These form a system of local parameters at $(0,0,0)$. Then $f=x=(x-y^2-z^2)+y^2+z^2$, so $f=f_1+f_2$, where $f_1(t_1,t_2,t_3)=t_1$, and $f_2(t_1,t_2,t_3)=t^2_2+t^2_3$. Also, the coefficient of $t_1$ in $f_1$ is positive (it is $1$), and $t_2, t_3$ do not appear in $f_1$. The quadratic form $f_2(0,t_2,t_3)=t^2_2+t^2_3$ is positive definite (when viewed as a quadratic form in the two variables $t_2, t_3$). So, according to the definition, $f$ satisfies the BHC at $(0,0,0)$. Here $f$ has a representation as follows:
$$f=\sigma_0+\sigma_1g_1+hf^2,$$
where $\sigma_0=y^2+z^2,\sigma_1=1,h=0$.
\end{example}

Now we define the {\it gradient ideal} of $f$ as follows:
$$I_{\grad}=\left\langle\frac{\partial f}{\partial x_1},\ldots,\frac{\partial f}{\partial x_n}\right\rangle.$$
Under the assumption that $I_{\grad}$ is radical, we have the following result.

\begin{theorem}[Nie-Demmel-Sturmfels \cite{NDS2006}]\label{NDS}
Suppose that
\begin{enumerate}
\item[(i)] $f\ge 0$ on $\mathbb R^n$,
\item[(ii)] $I_{\grad}$ is radical.
\end{enumerate}
Then $f$ is a sum of squares modulo $I_{\grad}$.
\end{theorem}

If we replace the radical condition of $I_{\grad}$ by an another condition that $f$ satisfies the BHC at each zero of $f$, then we will have the following result.

\begin{theorem}[Marshall \cite{Marshall2009}]\label{Marshall2}
Suppose that
\begin{enumerate}
\item[(i)] $f\ge 0$ on $\mathbb R^n$,
\item[(ii)] $f$ satisfies the BHC at each zero of $f$.
\end{enumerate}
Then $f$ is a sum of squares modulo $I_{\grad}$.
\end{theorem}

Similar to generalization of the gradient ideal, we define the {\it KKT ideal} of $f$ as follows:
$$I_{\KKT}=\langle F_1,\ldots,F_n,\lambda_1g_1,\ldots,\lambda_sg_s\rangle,$$
where $$F_i=\frac{\partial f}{\partial x_i}-\sum_{j=1}^s\lambda_j\frac{\partial g_j}{\partial x_i}, \forall i=1,\ldots,n.$$

Two following results are generalizations of theorem \ref{NDS} and theorem \ref{Marshall2} in the same way.

\begin{theorem}[Demmel-Nie-Powers \cite{NDP2007}]\label{DNP1}
Suppose that
\begin{enumerate}
\item[(i)] $f\ge 0$ on $K$,
\item[(ii)] $I_{\KKT}$ is radical.
\end{enumerate}
Then $f\in P+I_{\KKT}$.
\end{theorem}

\begin{theorem}[Hiep \cite{Dang2009}]\label{Dang}
Suppose that
\begin{enumerate}
\item[(i)] $f\ge 0$ on $K$,
\item[(ii)] $f$ satisfies the BHC at each zero of $f$ in $K$.
\end{enumerate}
Then $f\in P+I_{\KKT}$.
\end{theorem}

\begin{remark}\rm
The radical condition and the BHC condition are different. This means that there exist polynomials which satisfy the radical condition, but do not satisfy the BHC condition and conversely. The following example will demonstrate this difference.
\end{remark}

\begin{example}[Marshall \cite{Marshall2009}]\rm
\begin{itemize}
\item[]
\item[1.] Let $n=1$ and $s=0$ (so that $K=\mathbb R$). Then the polynomial in one variable $f(x)=6x^2+8x^3+3x^4$ satisfies the BHC condition, but it does not satisfy the radical condition. Indeed, $\displaystyle\frac{\partial f}{\partial x}=12x(x+1)^2$, $f(x)\ge 0$ on $\mathbb R$, $f$ has a zero at $x=0$, and $\displaystyle\frac{\partial^2f}{\partial x^2}(0)=12>0$. However, the gradient ideal $I=\langle 12x(x+1)^2\rangle$ which also is the KKT ideal, is not radical, because $g(x)=x(x+1)\in\sqrt{I}$, but $g\not\in I$.
\item[2.] Let $n=2$ and $s=0$ (so that $K=\mathbb R^2$). Then the polynomial in two variables $f(x,y)=x^2$ does not satisfy the BHC condition, but it satisfies the radical condition. Indeed, the Hessian matrix of $f$ is not positive definite at any zero of $f$ in $K$. However, the gradient ideal $I=\langle 2x\rangle$ which also is the KKT ideal, is radical.
\end{itemize}
\end{example}

\begin{remark}\rm
If we leave both the radical condition and the BHC condition, then we will have the corresponding representations of strictly positive polynomials.
\end{remark}

\begin{theorem}[Nie-Demmel-Sturmfels \cite{NDS2006}]
If $f>0$ on $\mathbb R^n$, then $f$ is a sum of squares modulo $I_{\grad}$.
\end{theorem}

\begin{theorem}[Demmel-Nie-Powers \cite{NDP2007}]\label{DNP2}
If $f>0$ on $K$, then $f\in P+I_{\KKT}$.
\end{theorem}

\begin{remark}\rm
In the proof of theorem \ref{DNP1}, theorem \ref{Dang} and theorem \ref{DNP2}, we must work in a larger polynomial ring $\mathbb R[x,\lambda]$, i.e., we must add the Lagrange multipliers to our representations.
\end{remark}

\section{Sums of squares modulo critical ideals}

In this section, we present our main results. These are similar to theorem \ref{DNP1} and theorem \ref{Dang}, but without modulo $I_{\KKT}$. It is replaced by modulo another ideal - the critical ideal of $f$ on $K$. In its proof, we work particularly in the polynomial ring $\mathbb R[x]$.

Let us start with some notations. The ideal
$$I(f,K):=\langle g_Jh_{J^c}, \forall J\subseteq\{1,\ldots,s\}\rangle$$
generated by $g_Jh_{J^c}$ is called the {\it critical ideal of $f$ on $K$}. By Proposition \ref{var2}, we have
$$C(f,K)=V^{\mathbb R}(I(f,K)).$$

\begin{theorem}\label{main1}
Suppose that
\begin{enumerate}
\item[(i)] $f\ge 0$ on $K$,
\item[(ii)] $f$ satisfies the BHC at each zero of $f$ in $K$.
\end{enumerate}
Then $f\in P+I(f,K)$.
\end{theorem}

To prove the theorem \ref{main1}, we need the following lemma.

\begin{lemma}\label{lemma1}
Let $W$ be an irreducible component of $V(I(f,K))$. If $W\cap\mathbb R^n\ne\emptyset$, then $f$ is constant on $W$.
\end{lemma}
\begin{proof}
This follows from the proof of lemma 3.6 in \cite{HP2010}.
\end{proof}

\begin{proof}[Proof of theorem \ref{main1}]
We decompose $V(I(f,K))$ into its irreducible components and let $W_0$ be the union of all the components whose intersection with $K$ is empty. We note that this includes all components $W$ with $W\cap \mathbb R^n=\emptyset$. Thus, by lemma \ref{lemma1}, $f$ is constant on each of the remaining components. We group together all components for which $f$ takes the same value. Then we have pairwise-disjoint subsets $W_1,\ldots,W_r$ of $W$ such that for each $i$, $f$ takes a constant value $a_i$ on $W_i$, with the $a_i$ being distinct. Further, since each $W_i$ contains a real point and $f$ is non-negative on $C(f,K)\cap K$, the value of $f$ on each $W_i$ is real and non-negative. We assume $a_1>\cdots>a_r\ge 0$. We fix a primary decomposition of $I(f,K)$, for each $i\in \{0,1...,r\}$, let $J_i$ be the intersection of those primary components corresponding to the irreducible components occurring in $W_i$. Thus, $V(J_i) = W_i, \forall i=0,1,\ldots,r$.

Since $W_i\cap W_j=\emptyset$, we have $J_i+J_j=\mathbb R[x]$ by theorem \ref{weaknullstellensatz}. Therefore the Chinese remainder theorem (see, e.g., \cite{Eisenbud1995}) implies that there is an isomorphism
$$\varphi: \mathbb R[x]/I(f,K)\longrightarrow \mathbb R[x]/J_0\times\mathbb R[x]/J_1\times\cdots\times\mathbb R[x]/J_r.$$

\begin{lemma}\label{lemma2}
There is $q_0\in P$ such that $f\equiv q_0 \bmod J_0$.
\end{lemma}
\begin{proof}
According to the argument presented above, $V(J_0)\cap K=\emptyset$, hence there exists $u_0\in P$ such that $-1\equiv u_0 \bmod J_0$. This result is a special case of theorem 8.6 in \cite{Lam1984}.

We write $f=f_1-f_2$ for SOS polynomials $f_1=(f+\frac{1}{2})^2$ and $f_2=(f^2+\frac{1}{4})$. Hence $f\equiv f_1+u_0f_2 \bmod J_0$. Let $q_0=f_1+u_0f_2\in P$. Then $f\equiv q_0\bmod J_0$.
\end{proof}

\begin{lemma}\label{lemma3}
$f$ is a sum of squares modulo $J_i$, for all $i=1,\ldots,r-1$.
\end{lemma}
\begin{proof}
According to the argument presented above, on each $W_i, 1\le i\le r-1, f=a_i>0$, and hence the polynomial $u=f/a_i-1$ vanishes on $W_i$. Then by theorem \ref{strongnullstellensatz} there exists some integer $k\ge 1$ such that $u^k\in J_i$. From the binomial identity, it follows that
$$1+u=\left(\sum_{j=0}^{k-1}\left(_{\ j}^{1/2}\right)u^j\right)^2+qu^k.$$
The reader can see clearly in lemma 7.24 in \cite{laurent2009}.\\
Thus $f=a_i(u+1)$ is a sum of squares modulo $J_i$.
\end{proof}

Now we continue the proof of theorem \ref{main1}.\\
If $a_r > 0$, then by the proof of lemma \ref{lemma3}, we imply that $f$ is a sum of squares modulo $J_r$.

\begin{lemma}\label{lemma4}
If $a_r=0$, then there is $q_r\in P$ such that $f\equiv q_r\bmod J_r$.
\end{lemma}
\begin{proof}
By the assumption that $f$ satisfies the BHC at each zero of $f$ on $K$ and by theorem \ref{Marshall}, there exist $g\in P$ and $h\in\mathbb R[x]$ such that $f=g+hf^2$, i.e., $f(1-hf)=g$. Since $f$ vanishes on $W_r$, $f^m\in J_r$ for some positive integer $m$. Let $t=hf,v=\displaystyle\sum_{i=0}^{m-1}t^i$. Then $t,v\in\mathbb R[x],t^m\in J_r$, and $(1-t)v\equiv 1 \bmod J_r$. By the binomial theorem, there exist $c_i\in\mathbb Q, i=0,1,\ldots,m-1$, such that
$$v\equiv \biggl(\sum_{i=0}^{m-1}c_it^i\biggr)^2\bmod J_r.$$
This yields $q_r\in P$ satisfying
$$f\equiv f(1-hf)v = gv\equiv q_r\bmod J_r.$$
\end{proof}

To finish the proof of theorem \ref{main1}, we claim the following lemma.

\begin{lemma}\label{lemma5}
Given $q_0,q_1,\ldots,q_r\in\mathbb R[x]$, there exists $q\in\mathbb R[x]$ such that $q-q_i\in J_i, \forall i=0,1,\ldots,r$. Moreover, if each $q_i\in P$, then $q\in P$.
\end{lemma}
\begin{proof}
The proof is by induction on $r \ge 1$. Assume $r = 1$. As $J_0+J_1 = \mathbb R[x]$, $1 = u_0+u_1$ for some $u_0\in J_0, u_1 \in J_1$. Set $q := u^2_0q_1+u^2_1q_0$; thus $q\in P$. Moreover, $q-q_0 = u^2_0q_1 +q_0(u^2_1 -1) = u^2_0q_1 - u_0(u_1 + 1)q_0 \in J_0$. Analogously, $q-q_1 \in J_1$. Let $t$ be the constructed polynomial, satisfying $t -q_0 \in J_0$ and $t-q_1 \in J_1$. Consider now the ideals $J_0\cap J_1, J_2,\ldots,J_r$. As $(J_0\cap J_1)+ J_i = \mathbb R[x] (i \ge 2)$, we can apply the induction assumption and deduce the existence of $q \in \mathbb R[x]$ for which $q-t \in J_0\cap J_1, q-q_i \in J_i (i \ge 2)$. Moreover, $q\in P$ if $t,q_2,...,q_r\in P$, which concludes the proof.
\end{proof}
Using lemma \ref{lemma2}, lemma \ref{lemma3}, lemma \ref{lemma4} and lemma \ref{lemma5}, we imply that there is $q\in P$ such that $f\equiv q \bmod I(f,K)$, i.e., $f\in P+I(f,K)$.
\end{proof}
\begin{remark}\rm
If we replace the BHC condition by the radical condition of $I(f,K)$, then we will have the following result.
\end{remark}
\begin{theorem}
Suppose that
\begin{enumerate}
\item[(i)] $f\ge 0$ on $K$,
\item[(ii)] $I(f,K)$ is radical.
\end{enumerate}
Then $f\in P+I(f,K)$.
\end{theorem}
\begin{proof}
From the proof of theorem \ref{main1}, by our definition of irreducibility, each $W_i$ is conjugate symmetric (i.e., a point $X\in\mathbb C^n$ belong to $W_i$ if and only if its complex conjugate $\bar X\in W_i$). By lemma 1 in \cite{NDS2006}, there exist polynomials $p_0,p_1,\ldots,p_r\in \mathbb R[x]$ such that $p_i(W_j)=\delta_{ij}$, where $\delta_{ij}$ is the Kronecker delta function.

We consider the polynomial $$q:=q_0p^2_0+\sum_{i=1}^ra_ip^2_i,$$
where $q_0$ is as in lemma \ref{lemma2}. By construction, $q\in P$.

Moreover, $f-q$ vanishes on $C(f,K)$, since $f(x)=q_0(x)=q(x)$ for $X\in W_0$ (by lemma \ref{lemma2}) and $f(x)=a_i=q(x)$ for $X\in W_i, \forall i=1,\ldots,r$.

By the assumption that $I(f,K)$ is radical and using Hilbert's Nullstellens\"{a}tz (see in \cite{Cox1997}), we deduce that $f-q\in I(f,K)$. This implies that $f\in P+I(f,K)$.
\end{proof}

\begin{remark}\rm
If we leave both the radical condition of $I(f,K)$ and the BHC condition, then we will have the corresponding representations of strictly positive polynomials.
\end{remark}
\begin{theorem}\label{main2}
If $f>0$ on $K$, then $f\in P+I(f,K)$.
\end{theorem}
\begin{proof}
 This follows similar argument in the proof of theorem \ref{main1}. However, we can assume $a_1>\cdots>a_r>0$. Thus, by lemma \ref{lemma3}, $f$ is a sum of squares modulo $J_i$, for all $i=1,\ldots,r$. Also by lemma \ref{lemma2} and lemma \ref{lemma5}, we imply that there is $q\in P$ such that $f\equiv q \bmod I(f,K)$, i.e., $f\in P+I(f,K)$.
\end{proof}

\section{Applications in optimization}
In this section, we present a result that is similar to theorem $4.1$ in \cite{NDP2007} and theorem $6.1$ in \cite{Dang2009}.

We consider the following optimization problem: Find
\begin{equation}\label{3}
f^*:=\inf_{x\in K}f(x).
\end{equation}
In the case where $K$ is compact, the SOS methods are based on representations of positive polynomials on compact semi-algebraic sets, which were presented in the theorems of Schm\"{u}dgen \cite{Schmudgen1991} and Putinar \cite{Putinar1993}. However, these theorems do not hold in the case where $K$ is not compact. A more traditional approach in numerical optimization methods uses the first order optimality conditions. Using theorem \ref{main1} and theorem \ref{main2}, we combine these two methods to give a procedure for approximating $f^*$ in the case where the semi-algebraic set is not necessarily compact.

In order to implement membership in $P+I(f,K)$ as a SDP, we need a bound on the degrees of the sums of squares involved. Thus, for $d\in\mathbb N$, we define the truncated preordering as follows:
$$P_d = \biggl\{\sum_{e\in\{0,1\}^s}\sigma_eg_1^{e_1}\ldots g_s^{e_s}\mid \deg(\sigma_eg_1^{e_1}\ldots g_s^{e_s})\le 2d\biggr\},$$
and the truncated critical ideal as follows:
$$I_d(f,K)=\biggl\{\sum_{J\subseteq\{1,\ldots,s\}}\phi_Jg_Jh_{J^c}\mid \deg(\phi_Jg_Jh_{J^c})\le 2d\biggr\}.$$
Then we define a sequence $\{f^*_d\}$ of SOS relaxations of the optimization problem (\ref{3}) as follows:
\begin{equation}\label{4}
f^*_d=\max_{\Gamma\in\mathbb R}\Gamma,
\end{equation}
\begin{equation}\label{5}
s.t. f(x)-\Gamma\in P_d+I_d(f,K).
\end{equation}
Obviously each $\Gamma$ feasible in (\ref{5}) is a lower bound of $f^*$. So $f^*_d\le f^*$. When we increase $d$, the feasible region defined by (\ref{5}) is increasing, and hence the sequence of lower bounds $\{f^*_d\}$ is also monotonically increasing. Thus we have
$$f^*_1\le f^*_2\le f^*_3\le\cdots\le f^*.$$
It can be shown that the sequence of lower bounds $\{f^*_d\}$ obtained from (\ref{4}) and (\ref{5}) converges to $f^*$ in (\ref{3}), provided that $f^*$ is attained at one point $x^*\in K$. We summarize in the following theorem:
\begin{theorem}\label{apply}
Assume $f$ has a minimum $f^* := f(x^*)$ at one point $x^*\in K$. Then $\lim_{d\to\infty} f^*_d = f^*$. Furthermore, if $f$ satisfies the BHC at each zero of $f-f^*$ in $K$, then there exists some $d\in\mathbb N$ such that $f^*_d = f^*$, i.e., the SOS relaxations (\ref{4}) and (\ref{5}) converge in a finite number of steps.
\end{theorem}
\begin{proof}
The proof is similar to that of theorem $6.1$ in \cite{Dang2009} (see also theorem $4.1$ in \cite{NDP2007}). However, we only consider the polynomial ring $\mathbb R[x]$.
\end{proof}

\subsection*{Acknowledgment}
The author would like to thank Prof. Murray Marshall, Prof. Ha Huy Vui and Assoc. Prof. Pham Tien Son for many interesting and helpful discussions on the topic of this work.

\end{document}